\UseRawInputEncoding
\documentclass[a4,12pt]{article}
\usepackage{amsmath}
\usepackage{amssymb}
\usepackage{amsthm}
\newtheorem{theorem}{Theorem}[section]
\newtheorem{lemma}[theorem]{Lemma}

\newtheorem{proposition}[theorem]{Proposition}
\theoremstyle{definition}
\newtheorem{example}[theorem]{Example}
\newtheorem{remark}[theorem]{Remark}
\newtheorem*{note}{Note}

\newtheorem*{corrections1}{Corrections to [11]}
\newtheorem*{corrections2}{Corrections to M. Igarashi, Duality relations among multiple series with three parameters, Tunisian J. Math. Vol. 2 (2020), no. 1, 217--236}

\begin{document}
\title{On the duality formula for parametrized multiple series}
\author{Masahiro Igarashi}
\date{}
\maketitle
\begin{abstract} 
We show that a duality formula for 
certain parametrized multiple series yields numerous relations among them. 
As a result, we obtain a new relation among extended multiple zeta values, which is an extension of Ohno's relation for multiple zeta values. 
We do the same study also to multiple Hurwitz zeta values, and 
obtain a new identity for them. 
\end{abstract}
\begin{flushleft}
\end{flushleft}
\begin{flushleft}
\end{flushleft}
\section{Introduction}
Fischler and Rivoal \cite{fr}, Kawashima \cite{kaw} and Ulanskii \cite{ul} studied the following extension of 
the multiple zeta value (MZV for short):
\begin{equation}
\sum_{\begin{subarray}{c}0<m_1<_{c_1}\cdots<_{c_{p-1}}m_p<\infty\\
m_i\in\mathbb{Z}\end{subarray}}
\frac{1}{m_{1}^{k_1}{\cdots}m_{p}^{k_p}},
\end{equation}
where $1{\le}p\in\mathbb{Z}$; $c_i\in\{0,1\}$, $1{\le}k_i\in\mathbb{Z}$ ($i=1,\ldots,p-1$), $2{\le}k_p\in\mathbb{Z}$; the symbols $<_{c_i}$ ($i=1,\ldots,p-1$) denote $<$ if $c_i=1$ and $\le$ if $c_i=0$. 
For $c_i=1$ ($i=1,\ldots,p-1$), this multiple series becomes MZV, 
which was studied in Euler \cite{eu}, Hoffman \cite{ho} and Zagier \cite{z}. 
MZV has rich mathematical contents. 
For example, the product of MZVs has two kinds of multiplication laws, 
one is induced from the series expression of MZV and the other an iterated integral expression, and they yield numerous relations over $\mathbb{Q}$. Furthermore, the set of MZVs generates a $\mathbb{Q}$-algebra with the multiplication laws. 
Understanding the algebra is one of the important problems in mathematics, because MZV has connection with various mathematical objects, e.g., 
knot invariants, Feynman integrals, modular forms and mixed Tate motives over $\mathbb{Z}$. The extension (1) keeps these rich contents. 
Another simple case $c_i=0$ ($i=1,\ldots,p-1$) becomes the multiple zeta-star value (MZSV for short). MZSV is an interesting object itself. Indeed, it is almost the same as MZV in appearance, but aspect of its relations is fairly different: see, e.g., \cite{ikoo}, \cite{ow} and \cite{oz}. 
It can be seen that relations among MZSVs have brevity. 
All other cases of (1) are hybrids of both $<$ and $\le$. 
(They can be expressed in $\mathbb{Z}$-linear combinations of MZVs or MZSVs.) 
These hybrid MZVs frequently appear in the study of MZV and MZSV. 
The most important point of the extension (1) is that it gives a unified expression of the three objects MZV, MZSV and the hybrid MZV, and this gives unified extensions of relations among MZVs and MZSVs: see the case $\alpha=1$ of (7) below. 
Hereafter we call the multiple series (1) the extended multiple zeta value (EMZV for short). 
Kawashima \cite{kaw} introduced EMZV to study a Newton series and relations among MZVs. He proved a duality relation among EMZVs \cite[Proposition 5.3]{kaw}. 
Ulanskii \cite{ul} gave an algebraic formulation for EMZV, and proved some basic properties of EMZV: 
the Chen iterated integral representation, the duality formula, the shuffle and stuffle relations \cite[Corollary 2, Theorems 1, 2 and 3]{ul}. 
Fischler and Rivoal \cite{fr} used EMZV and an extended multiple polylogarithm to study a Pad\'{e} approximation problem involving multiple polylogarithms. They also proved a duality formula for EMZVs, and applied it to construction of $\mathbb{Q}$-linear forms in the Riemann zeta values (see \cite[Section 2]{fr}). We note that multiple series of the extended form (1) naturally appear as derivatives of hypergeometric series (see \cite{i2018}). 
\par 
In the present paper, we also study multiple series of the extended form (1). 
Our interest is duality relations among them. In the former part of the paper, we study the multiple series 
\begin{equation}
\begin{aligned}
\sum_{\begin{subarray}{c}0{\le}m_1<_{c_1}\cdots<_{c_{p-1}}m_p<\infty\\
m_i\in\mathbb{Z}\end{subarray}}
\frac{(\alpha)_{m_1}}{{m_1}!}\frac{{m_p}!}{(\alpha)_{m_p}}
\left\{\prod_{i=1}^{p}\frac{1}{(m_i+\alpha)^{a_i}(m_i+\beta)^{b_i}}\right\},
\end{aligned}
\end{equation}
where $1{\le}p\in\mathbb{Z}$; $a_{i}, b_{i}\in\mathbb{Z}$ such that $a_{i}+b_{i}\ge1$ ($i=1,\ldots,p-1$), $a_{p}+b_{p}\ge2$; 
$\alpha,\beta\in\mathbb{C}$ such that $\mathrm{Re}(\alpha)>0$, $\beta\notin\mathbb{Z}_{\le0}:=\{0,-1,-2,\ldots\}$; $(a)_m$ denotes the Pochhammer symbol, i.e., $(a)_m=a(a+1)\cdots(a+m-1)$ $(1{\le}m\in\mathbb{Z})$ and 
$(a)_0=1$. 
The Pochhammer symbol can be expressed as a quotient of the gamma function: 
$(a)_m=\Gamma(a+m)/\Gamma(a)$. Therefore Stirling's formula for $\Gamma(z)$ can be applied to the estimation of $(a)_m$. 
For the convergence of (2), see \cite[Lemma 2.1]{i2009}. 
The multiple series (2) is an extension of both EMZV and 
our two-parameter multiple series studied in \cite{ig2007} and \cite{i2009}. 
The study of this kind of two-parameter extension of MZV was originated by the author in \cite{ig2007}. 
See also \cite{i2009}, \cite{i2018} and \cite[Note 2]{i2020}. 
The author proved also the cyclic sum formula for 
the case $c_i=1$ ($i=1,\ldots,p-1$) of (2) and for the case $c_i=0$ ($i=1,\ldots,p-1$): see \cite{i202206} and also \cite[Note 2 (iv) and (v)]{i2020}. 
\subsection{Definitions and notation}
To describe our results concisely, we follow Ulanskii's algebraic formulation for EMZV 
\cite{ul}, which is an extension of Hoffman's for MZV \cite{ho2}. 
The formulation is done by using three non-commutative variables, $x_0$, $x_1$ and $x_{-1}$. Hereafter we assume that 
$i,m,n,p,q,r,l_i,m_i,k_i,k^{'}_i,s_i,r_i,y^{(i)}_j,M^{(i)}_{j}\in\mathbb{Z}$. 
For brevity, we frequently use the notation 
$z_{c_{i-1}}(k_i):=x_{c_{i-1}}x_{-1}^{k_i-1}$ 
($c_{i-1}\in\{0,1\}$, $k_i\ge1$). For example, we write $x_{1}x_{-1}^{k_1-1}x_{c_1}x_{-1}^{k_2-1}{\cdots}x_{c_{p-1}}x_{-1}^{k_p-1}$ as
\begin{equation*} 
z_1(k_1)z_{c_1}(k_2){\cdots}z_{c_{p-1}}(k_p)=\prod_{i=1}^{p}z_{c_{i-1}}(k_i), 
\end{equation*}
where $c_0=1$. Here we put
\begin{equation*}
B^0:=\left\{\prod_{i=1}^{p}z_{c_{i-1}}(k_i)\,\Biggl|\,
p\ge0, c_0=1, c_i\in\{0,1\}, k_i\ge1 (i=1,\ldots,p-1), k_p\ge2\right\},
\end{equation*}
where $\prod_{i=1}^{0}z_{c_{i-1}}(k_i)=1\in\mathbb{Q}$, and denote by $V^0$ the $\mathbb{Q}$-vector space whose basis is $B^0$. 
(This vector space corresponds to $Y^0$ in \cite{ul}.) We define 
the evaluation map $Z=Z_{(\alpha,\beta)}: B^0 \rightarrow \mathbb{C}$ by $Z(1;(\alpha,\beta))=1$ and 
\begin{equation}
\begin{aligned}
&Z(z_1(k_1)z_{c_1}(k_2){\cdots}z_{c_{p-1}}(k_p);(\alpha,\beta))\\
=&\sum_{\begin{subarray}{c}0{\le}m_1<_{c_1}\cdots<_{c_{p-1}}m_p<\infty\end{subarray}}
\frac{(\alpha)_{m_1}}{{m_1}!}\frac{{m_p}!}{(\alpha)_{m_p+1}}
\left\{\prod_{i=1}^{p-1}\frac{1}{(m_i+\beta)^{k_i}}\right\}\frac{1}{(m_p+\beta)^{k_p-1}},
\end{aligned}
\end{equation}
where $\alpha,\beta\in\mathbb{C}$ such that $\mathrm{Re}(\alpha)>0$, $\beta\notin\mathbb{Z}_{\le0}$. 
This map can be extended to a $\mathbb{Q}$-linear map onto the whole space $V^0$. 
To describe partial derivatives of (3), we use the evaluation map 
$Z^{*}_{(\{r_i\}_{i=1}^{q})}=Z^{*}_{(\{r_i\}_{i=1}^{q}), (\beta,\alpha)}:B^0 \rightarrow \mathbb{C}$ defined by $Z^{*}_{(\{r_i\}_{i=1}^{q})}(1;(\beta,\alpha))=1$ 
and 
\begin{equation}
\begin{aligned}
&Z^{*}_{(\{r_i\}_{i=1}^{q})}(z_1(k_1)z_{c_1}(k_2){\cdots}z_{c_{q-1}}(k_q);(\beta,\alpha))\\
=
&\sum_{\begin{subarray}{c}0{\le}m_1<_{c_1}M_{1}^{(2)}{\le}\cdots{\le}M_{r_2}^{(2)}<_{1-c_2}m_2\\
\vdots\\
m_{i-1}<_{c_{i-1}}M_{1}^{(i)}{\le}\cdots{\le}M_{r_i}^{(i)}<_{1-c_i}m_i\\
\vdots\\
m_{q-1}<_{c_{q-1}}M_{1}^{(q)}{\le}\cdots{\le}M_{r_q}^{(q)}<_{1-c_q}m_q<\infty
\end{subarray}}
\frac{(\beta)_{m_1}}{{m_1}!}
\frac{{m_q}!(m_q+\alpha)}{(\beta)_{m_q+1}}\\
&\times\frac{1}{(m_1+\beta)^{r_1}(m_1+\alpha)^{k_1}}
\left\{\prod_{i=2}^{q}\left(\prod_{j=1}^{r_i}\frac{1}{M_{j}^{(i)}+\beta}\right)
\frac{1}{(m_i+\alpha)^{k_i}}\right\},
\end{aligned}
\end{equation}
where $q\ge1$, $r_i\ge0$ ($i=1,\ldots,q$), $c_q=1$, $\alpha,\beta\in\mathbb{C}$ such that $\alpha\notin\mathbb{Z}_{\le0}$, $\mathrm{Re}(\beta)>0$. 
This map can also be extended to a $\mathbb{Q}$-linear map onto the whole space $V^0$. If $r_i=0$, we regard the inequalities $m_{i-1}<_{c_{i-1}}M_{1}^{(i)}{\le}\cdots{\le}M_{r_i}^{(i)}<_{1-c_i}m_i$ of (4) as $m_{i-1}<_{c_{i-1}}m_i$. 
For $r_i=0$ ($i=1,\ldots,q$) and for $q=1$, the multiple series (4) becomes
\begin{equation*}
\begin{aligned}
&Z^{*}_{(\{0\}_{i=1}^{q})}\left(\prod_{i=1}^{p}z_{c_{i-1}}(k_i);(\beta,\alpha)\right)
=Z\left(\prod_{i=1}^{p}z_{c_{i-1}}(k_i);(\beta,\alpha)\right),\\
&Z^{*}_{(r_1)}(z_1(k_1);(\beta,\alpha))
=
\sum_{\begin{subarray}{c}0{\le}m_1<\infty
\end{subarray}}
\frac{1}{(m_1+\beta)^{r_1+1}(m_1+\alpha)^{k_1-1}},
\end{aligned}
\end{equation*}
respectively; therefore the map $Z^{*}_{(\{r_i\}_{i=1}^{q})}$ is an extension of the map $Z$ with the additional parameters $\{r_i\}_{i=1}^{q}$. 
This gives an algebraic description of partial derivatives on $\alpha$ of (3): 
see the proof of Theorem 1.1 (i). 
To describe another derivation aspect of our results, 
we need also the maps $\sigma_r^{b,1}$, $\sigma_r:B^0 \rightarrow V^0$ defined by $\sigma_r^{b,1}(1)=\sigma_r(1)=1$ and 
\begin{equation*}
\begin{aligned}
&\sigma_r^{b,1}(z_1(k_1)z_{c_1}(k_2){\cdots}z_{c_{p-1}}(k_p))\\
=
&\sum_{\begin{subarray}{c}r_1+\cdots+r_p=r\\
r_i\ge0\end{subarray}}
\left\{\prod_{i=1}^{p-1}\binom{k_{i}+r_{i}-1}{r_{i}}\right\}
\binom{k_{p}+r_p-2}{r_{p}}
\prod_{i=1}^{p}z_{c_{i-1}}(k_i+r_i),\\
&\sigma_r(z_1(k_1)z_{c_1}(k_2){\cdots}z_{c_{p-1}}(k_p))\\
=
&\sum_{\begin{subarray}{c}\sum_{i=1}^{p-1}c_{i}r_i+r_{p}=r\\
c_{i}r_i, r_{p}\ge0\end{subarray}}
\left\{\prod_{i=1}^{p-1}z_{c_{i-1}}(k_i+c_{i}r_i)\right\}z_{c_{p-1}}(k_p+r_p),
\end{aligned}
\end{equation*}
where $r\ge0$. 
These are variations of the map $\sigma_m$ used in \cite[Section 6]{ikz}. (For $c_i=1$ ($i=1,\ldots,p-1$), the map $\sigma_r$ becomes $\sigma_m$.) In the present paper, we use the following standard definition of the dual: 
Let $\tau$ be the map $\tau:B^0 \rightarrow B^0$ defined by $\tau(1)=1$ and 
$\tau(x_{1}x_{e_1}{\cdots}x_{e_{n-1}}x_{-1})=x_{1}x_{-e_{n-1}}{\cdots}x_{-e_1}x_{-1}$, 
where $n\ge1$ and $e_i\in\{-1,0,1\}$ ($i=1,\ldots,n-1$). 
Then $\tau(v)$ is called the dual of $v$. 
It is obvious that $\tau^2(v)=v$. The maps $\sigma_r^{b,1}$, $\sigma_r$ and $\tau$ can be extended to 
$\mathbb{Q}$-linear maps from the whole space $V^0$ to itself. 
Let $v\in{B^0}$. Then $\tau(v)$ can also be expressed as $\tau(v)=\prod_{i=1}^{q}z_{c^{'}_{i-1}}(k^{'}_i)$, where $q\ge0$, $c^{'}_0=1$, $c^{'}_i\in\{0,1\}$, $k^{'}_{i}\ge1$ ($i=1,\ldots,q-1$), 
$k^{'}_q\ge2$. Hereafter we assume this expression for $\tau(v)$. 
For any fixed real numbers $a$, $b$ ($a<b{\le}\infty$), we regard 
the sum $\sum_{a<_{c_1}M_1<_{c_2}\cdots<_{c_{p}}M_p<_{c_{p+1}}b}A_{M_1,\ldots,M_p}$ as 1 if $p=0$. 
\subsection{Main theorem}
For $v=\prod_{i=1}^{p}z_{1}(k_i)$, we proved in \cite{i2009} a large class of 
relations among the multiple series $Z(v;(\alpha,\alpha))$ by using the duality formula
\begin{equation}
Z(v;(\alpha,\beta))=Z(\tau(v);(\beta,\alpha)), 
\quad \mathrm{Re}(\alpha), \mathrm{Re}(\beta)>0 
\end{equation}
(see \cite[Theorem 1.1 and Lemma 2.3]{i2009}). 
The symmetry on $\alpha$ and $\beta$ of (5) played an essential role for the proof (see \cite[Section 2]{i2009}). 
From this fact, we think that this kind of symmetry 
of parametrized multiple series is useful for the study of relations among multiple series (e.g., MZVs). 
In the present paper, we prove a large class of relations among (2) by using a duality formula for (3), which is an extension of (5) (see Lemma 2.2 below), 
that is, we prove the following:
\begin{theorem}
Let $v\in{B^0}$, and let $\tau(v)$ be its dual. Then 
\par
$(i)$
\begin{equation}
\begin{aligned}
Z(\sigma^{b,1}_{r}(v);(\alpha,\beta))
=
&\sum_{\begin{subarray}{c}c^{'}_{1}r_1+\sum_{i=2}^{q}r_i=r\\
c^{'}_{1}r_1, r_i\ge0\end{subarray}}
Z^{*}_{(c^{'}_1r_1,\{r_i\}_{i=2}^{q})}(\tau(v);(\beta,\alpha))
\end{aligned}
\end{equation}
for all $r\ge0$, $\alpha,\beta\in\mathbb{C}$ with $\mathrm{Re}(\alpha), \mathrm{Re}(\beta)>0$, 
where $c^{'}_1$ and $q$ are those of the dual 
$\tau(v)=\prod_{i=1}^{q}z_{c^{'}_{i-1}}(k^{'}_i)$.
\par
$(ii)$
\begin{equation}
Z(\sigma_{r}(v);\alpha)=Z(\sigma_{r}\tau(v);\alpha)
\end{equation}
for all $r\ge0$, $\alpha\in\mathbb{C}$ with $\mathrm{Re}(\alpha)>0$, where 
$Z(v;\alpha):=Z(v;(\alpha,\alpha))$ $(v\in{B^0}$, $\alpha\in\mathbb{C}$ with 
$\mathrm{Re}(\alpha)>0)$.
\end{theorem}
The identities (6) and (7) yield numerous relations among (2) and EMZVs. 
In fact, even the simple case $v=\prod_{i=1}^{p}z_1(k_i)$ of (7) becomes 
a large class of relations \cite[Theorem 1.1]{i2009}. 
We note that, for $v\in{B^0}$ and its dual $\tau(v)$, the identity (6) gives two different relations among (2): see Example 2.7 (i) below. 
This is one of the features of our multi-parameter extension. 
The specialization $\alpha=\beta=1$ of (6) and of (7) give a large class of relations among EMZVs. In particular, the case $\alpha=1$ of (7) is a new extension of Ohno's relation for MZVs \cite[Theorem 1]{o}, 
which extends Ohno's relation to a relation involving MZSVs and the hybrid MZVs. Besides this, the specialization $v=z_1(k_1)\left\{\prod_{i=2}^{p}z_{0}(k_i)\right\}$ and $\alpha=1$ of (7) gives a relation between MZSVs and the hybrid MZVs: see (26) below. 
\par 
In the present paper, we shall study also a duality formula for multiple Hurwitz zeta values; see Section 3 for details. The result, Theorem 3.1 below, is another main theorem of the paper. 
\par
We explain our idea for the proofs of Theorems 1.1 and 3.1. 
It is to use symmetries on $\alpha$ and $\beta$ of the multiple series (2), (27) and (28) below; see (9) and (32) below. 
These symmetries can be found by making a change of variables to 
an iterated integral representation of (2) and of (27): see the proof of (9) and of (32). 
We note that the change of variables also brings changes of the positions of the parameters of (2), (27) and (28): compare the positions of 
the parameters $\alpha$ and $\beta$ on both sides of (9) and of (32). 
This plays an essential role for deriving various and numerous relations. 
Indeed, the changes of the positions allow us to show that partial differential operators act 
on each side of (9) and of (32) in two different ways, and this gives the relations 
in the theorems. (The above idea was used also in our previous works \cite{ig2007} and \cite{i2009}.) 
Another main tool for the proof is our calculus of the Pochhammer symbol 
$(a)_m$ used in \cite{i2018}, which was developed in its preprints 
distributed in 2013. 
This allows us to calculate derivatives of $(a)_m$ without calculating products of finite multiple harmonic sums: see, e.g., (10)--(12) below and compare them with our calculus used in \cite{i2011}. 
\par
We shall prove Theorem 1.1 in Section 2. Our method of the proof is similar to that in \cite{ig2007} and \cite{i2009}. 
In Section 3, we shall apply our method also to multiple Hurwitz zeta values, and shall obtain an identity for them similar to Theorem 1.1 (i), 
which also yields numerous relations. 
\par 
The present paper is a revised version of preprints of mine which were 
distributed in October 2015. An earlier version of the paper was submitted to a journal in June 2016. See also Note at the end of Section 3 and 
\cite[Note 2 (iii)--(v)]{i2020}. 
\section{Proof of Theorem 1.1}
We first prove a duality formula for (3). 
We define the symbol $\omega_{e_i}(t)$ ($e_i\in\{-1,0,1\}$) by 
\begin{equation*}
\omega_{-1}(t)=\frac{1}{t}, \quad \omega_{0}(t)=\frac{1}{t(1-t)}, \quad \omega_{1}(t)=\frac{1}{1-t}.
\end{equation*}
\begin{lemma}
Let $n\ge1$ and $e_i\in\{-1,0,1\}$ $(i=1,\ldots,n-1)$. Then the following iterated integral representation of $(3)$ holds$:$
\begin{equation}
\begin{aligned}
&Z(x_{1}x_{e_1}{\cdots}x_{e_{n-1}}x_{-1};(\alpha,\beta))\\
=
&\idotsint\displaylimits_{\begin{subarray}{c}
0<t_0<\cdots<t_n<1
\end{subarray}}
(1-t_0)^{1-\alpha}t_0^{\beta-1}\omega_1(t_0)
\left\{\prod_{i=1}^{n-1}\omega_{e_i}(t_i)\right\}
\omega_{-1}(t_n)t_n^{1-\beta}(1-t_n)^{\alpha-1}
\mathrm{d}t_{0}\cdots\mathrm{d}t_{n}
\end{aligned}
\end{equation}
for all $\alpha,\beta\in\mathbb{C}$ with $\mathrm{Re}(\alpha), \mathrm{Re}(\beta)>0$.
\end{lemma}
\begin{proof}
The proof is the same as that in \cite[Proof of Lemma 2.2]{i2009}. 
The integrand of the iterated integral of (8) can be rewritten as 
\begin{equation*}
\omega_1
\left\{\prod_{i=1}^{n-1}\omega_{e_i}\right\}
\omega_{-1}
=
\prod_{i=1}^{p}
\omega_{c_{i-1}}\omega_{-1}^{k_i-1},
\end{equation*}
where $p\ge1$, $c_0=1$, $c_i\in\{0,1\}$, $k_i\ge1$ ($i=1,\ldots,p-1$), $k_p\ge2$. 
Using this expression, we can rewrite the iterated integral of (8), which we denote by $I$, as 
\begin{equation*}
\begin{aligned}
I
=
&\idotsint\displaylimits_{\begin{subarray}{c}
0<t_{11}<\cdots<t_{1k_1}<\\
\vdots\\
<t_{i1}<\cdots<t_{ik_i}<\\
\vdots\\
<t_{p1}<\cdots<t_{pk_p}<1
\end{subarray}}
(1-t_{11})^{1-\alpha}t_{11}^{\beta-1}
\left\{
\prod_{\begin{subarray}{c}i=1\end{subarray}}^{p}
\omega_{c_{i-1}}(t_{i1})
\left(\prod_{j=2}^{k_i}\omega_{-1}(t_{ij})\right)
\right\}
t_{pk_p}^{1-\beta}(1-t_{pk_p})^{\alpha-1}\\
&\times\left(\prod_{i=1}^{p}\prod_{j=1}^{k_i}\mathrm{d}t_{ij}\right).
\end{aligned}
\end{equation*}
Further, applying the expansions 
\begin{equation*}
\begin{aligned}
\left(1-t_{11}\right)^{-\alpha}
=\sum_{m=0}^{\infty}\frac{(\alpha)_{m}}{m!}t^{m}_{11}, \quad 
\left(1-t_{i1}\right)^{-1}
=\sum_{m=0}^{\infty}t^{m}_{i1}
\end{aligned}
\end{equation*}
($i=2,\ldots,p$) to the integrand and integrating term by term, we have also 
the identities
\begin{equation*}
\begin{aligned}
I
=
&\sum_{\begin{subarray}{c}0{\le}m_1<_{c_1}\cdots<_{c_{p-1}}m_p<\infty\end{subarray}}
\frac{(\alpha)_{m_1}}{{m_1}!}\left\{\prod_{i=1}^{p-1}\frac{1}{(m_i+\beta)^{k_i}}\right\}\frac{1}{(m_p+\beta)^{k_p-1}}\\
&\times\int_{0}^{1}(1-t_{pk_p})^{\alpha-1}{t^{m_p}_{pk_p}}\mathrm{d}t_{pk_p}\\
=
&\sum_{\begin{subarray}{c}0{\le}m_1<_{c_1}\cdots<_{c_{p-1}}m_p<\infty\end{subarray}}
\frac{(\alpha)_{m_1}}{{m_1}!}\left\{\prod_{i=1}^{p-1}\frac{1}{(m_i+\beta)^{k_i}}\right\}\frac{1}{(m_p+\beta)^{k_p-1}}\\
&\times\frac{\Gamma(\alpha)\Gamma(m_p+1)}
{\Gamma(\alpha+m_p+1)}\\
=
&\sum_{\begin{subarray}{c}0{\le}m_1<_{c_1}\cdots<_{c_{p-1}}m_p<\infty\end{subarray}}
\frac{(\alpha)_{m_1}}{{m_1}!}\frac{{m_p}!}{(\alpha)_{m_p+1}}
\left\{\prod_{i=1}^{p-1}\frac{1}{(m_i+\beta)^{k_i}}\right\}\frac{1}{(m_p+\beta)^{k_p-1}}\\
=
&Z(x_1x_{-1}^{k_1-1}x_{c_1}x_{-1}^{k_2-1}{\cdots}x_{c_{p-1}}x_{-1}^{k_p-1};(\alpha,\beta))
\end{aligned}
\end{equation*}
for $\alpha,\beta\in\mathbb{C}$ with $\mathrm{Re}(\alpha), \mathrm{Re}(\beta)>0$. 
The monomial $x_1x_{-1}^{k_1-1}x_{c_1}x_{-1}^{k_2-1}{\cdots}x_{c_{p-1}}x_{-1}^{k_p-1}$ can be rewritten as 
$x_{1}x_{e_1}{\cdots}x_{e_{n-1}}x_{-1}$ ($e_i\in\{-1,0,1\}$); thus we obtain (8).
\end{proof}
Using Lemma 2.1, we can prove the following duality formula 
for (3), which will play an essential role for the proof of Theorem 1.1:
\begin{lemma}[Duality formula]
Let $v\in{B^0}$, and let $\tau(v)$ be its dual. Then 
\begin{equation}
Z(v;(\alpha,\beta))
=
Z(\tau(v);(\beta,\alpha))
\end{equation}
for all $\alpha,\beta\in\mathbb{C}$ with $\mathrm{Re}(\alpha), \mathrm{Re}(\beta)>0$.
\end{lemma}
\begin{proof} 
Making the change of variables $t_i=1-u_{n-i}$ ($i=0,1,\ldots,n$; see \cite[p.~510]{z}) to 
the iterated integral on the right-hand side of (8), we obtain (9).
\end{proof}
\begin{remark}
Ulanskii \cite{ul} proved the Chen iterated integral representation and the duality formula for EMZVs \cite[Corollary 2 and Theorem 1]{ul}. The case $\alpha=\beta=1$ of (8) and of (9) are Corollary 2 and Theorem 1 of \cite{ul}, respectively. We note also that (8) and (9) are extensions of our previous results \cite[Lemmas 2.2 and 2.3]{i2009}. 
\end{remark}
\subsection{Proof of Theorem 1.1 (i)}
Let $v=\prod_{i=1}^{p}z_{c_{i-1}}(k_i)\in{B^0}$, 
and let $\tau(v)=\prod_{i=1}^{q}z_{c^{'}_{i-1}}(k^{'}_i)$ be its dual. 
The proof is done by differentiating both sides of the duality formula (9) $r$ times with respect to $\beta$. 
The left-hand side of (6) is an immediate result of differentiating that of (9). To obtain the right-hand side, we calculate the derivative 
$\frac{(-1)^r}{r!}\frac{\mathrm{d}^r}{\mathrm{d}\beta^r}\left(\frac{(\beta)_{m_1}}{(\beta)_{m_q+1}}\right)$ ($r\ge0$). From the definition of $c_i$ ($=0,1$), we have $(\beta)_{m_1+c^{'}_1}=(\beta)_{m_1}(m_1+\beta)^{c^{'}_1}$. Using this, we have the expression 
\begin{equation}
\frac{(\beta)_{m_1}}{(\beta)_{m_q+1}}
=
\frac{1}{(m_1+\beta)^{c^{'}_1}}
\left(\prod_{i=2}^{q}\frac{(\beta)_{m_{i-1}+c^{'}_{i-1}}}{(\beta)_{m_{i}+c^{'}_i}}\right),
\end{equation}
where $m_1,\ldots,m_q\in\mathbb{Z}$ such that $0{\le}m_1{<_{c^{'}_1}}\cdots{<_{c^{'}_{q-1}}}m_q$ and $c^{'}_q=1$. The derivatives of the factors on the right-hand side of (10) can be calculated as follows: 
\begin{equation}
\begin{aligned}
&\frac{(-1)^{r_i}}{{r_i}!}\frac{\mathrm{d}^{r_i}}{\mathrm{d}\beta^{r_i}}
\left(\frac{(\beta)_{m_{i-1}+c^{'}_{i-1}}}{(\beta)_{m_{i}+c^{'}_{i}}}\right)\\
=
&\frac{(\beta)_{m_{i-1}+c^{'}_{i-1}}}{(\beta)_{m_{i}+c^{'}_{i}}}
\sum_{\begin{subarray}{c}m_{i-1}+c^{'}_{i-1}{\le}M_1^{(i)}{\le}\cdots{\le}M_{r_i}^{(i)}<m_{i}+c^{'}_{i}\end{subarray}}
\prod_{j=1}^{r_i}\frac{1}{M^{(i)}_{j}+\beta}\\
=
&\frac{(\beta)_{m_{i-1}+c^{'}_{i-1}}}{(\beta)_{m_{i}+c^{'}_{i}}}
\sum_{\begin{subarray}{c}m_{i-1}{<_{c^{'}_{i-1}}}M_1^{(i)}{\le}\cdots{\le}M_{r_i}^{(i)}{<_{1-c^{'}_i}}m_{i}\end{subarray}}
\prod_{j=1}^{r_i}\frac{1}{M^{(i)}_{j}+\beta}
\end{aligned}
\end{equation}
for $r_i\ge0$ ($i=2,\ldots,q$). (From the definitions of the symbols $c_{i}$ 
and $<_{c_i}$, the inequalities $m_{i-1}+c^{'}_{i-1}{\le}M_1^{(i)}$ and $M_{r_i}^{(i)}<m_{i}+c^{'}_{i}$ of (11) can be rewritten as 
$m_{i-1}{<_{c^{'}_{i-1}}}M_1^{(i)}$ and 
$M_{r_i}^{(i)}{<_{1-c^{'}_i}}m_{i}$, respectively.) Using (10) and (11), 
we have 
\begin{equation}
\begin{aligned}
&\frac{(-1)^r}{r!}\frac{\mathrm{d}^r}{\mathrm{d}\beta^r}\left(\frac{(\beta)_{m_1}}{(\beta)_{m_q+1}}\right)\\
=
&\sum_{\begin{subarray}{c}r_1+\cdots+r_q=r\\
r_i\ge0\end{subarray}}
\frac{\binom{r_1+c^{'}_1-1}{r_1}}{(m_1+\beta)^{r_1+c^{'}_{1}}}\\
&\times\left(\prod_{i=2}^{q}\frac{(\beta)_{m_{i-1}+c^{'}_{i-1}}}{(\beta)_{m_{i}+c^{'}_{i}}}
\sum_{\begin{subarray}{c}m_{i-1}{<_{c^{'}_{i-1}}}M_1^{(i)}{\le}\cdots{\le}M_{r_i}^{(i)}{<_{1-c^{'}_i}}m_{i}\end{subarray}}
\prod_{j=1}^{r_i}\frac{1}{M^{(i)}_{j}+\beta}\right)\\
=
&\sum_{\begin{subarray}{c}r_1+\cdots+r_q=r\\
r_i\ge0\end{subarray}}
\frac{\binom{r_1+c^{'}_{1}-1}{r_1}}{(m_1+\beta)^{r_1}}\\
&\times\frac{(\beta)_{m_1}}{(\beta)_{m_q+1}}
\prod_{i=2}^{q}\left(
\sum_{\begin{subarray}{c}m_{i-1}{<_{c^{'}_{i-1}}}M_1^{(i)}{\le}\cdots{\le}M_{r_i}^{(i)}{<_{1-c^{'}_i}}m_{i}\end{subarray}}
\prod_{j=1}^{r_i}\frac{1}{M^{(i)}_{j}+\beta}\right)\\
=
&\sum_{\begin{subarray}{c}c^{'}_{1}r_1+\sum_{i=2}^{q}r_i=r\\
c^{'}_{1}r_1, r_i\ge0\end{subarray}}
\frac{1}{(m_1+\beta)^{c^{'}_{1}r_1}}\\
&\times\frac{(\beta)_{m_1}}{(\beta)_{m_q+1}}
\prod_{i=2}^{q}\left(
\sum_{\begin{subarray}{c}m_{i-1}{<_{c^{'}_{i-1}}}M_1^{(i)}{\le}\cdots{\le}M_{r_i}^{(i)}{<_{1-c^{'}_i}}m_{i}\end{subarray}}
\prod_{j=1}^{r_i}\frac{1}{M^{(i)}_{j}+\beta}\right)
\end{aligned}
\end{equation}
for $r\ge0$: the last identity of (12) follows from the identity
\begin{equation}
\binom{r_i+c^{'}_{i}-1}{r_i} = \left\{\begin{alignedat}{3}
                    &1&\quad&\text{if $c^{'}_{i}=r_i=0$ or $c^{'}_{i}=1$, $r_i\ge0$},\\
                    &0&\quad&\text{if $c^{'}_{i}=0$, $r_i\ge1$}.
                     \end{alignedat}\right.
\end{equation}
Therefore, using (12), we have 
\begin{equation}
\begin{aligned}
&\frac{(-1)^r}{r!}\frac{\partial^r}{\partial\beta^r}\left(\frac{(\beta)_{m_1}}{m_1!}\frac{m_q!}{(\beta)_{m_q+1}}
\left\{\prod_{i=1}^{q-1}\frac{1}{(m_i+\alpha)^{k^{'}_i}}\right\}\frac{1}{(m_q+\alpha)^{k^{'}_q-1}}\right)\\
=
&\sum_{\begin{subarray}{c}c^{'}_{1}r_1+\sum_{i=2}^{q}r_i=r\\
c^{'}_{1}r_1, r_i\ge0\end{subarray}}
\sum_{\begin{subarray}{c}m_1{<_{c^{'}_1}}M_{1}^{(2)}{\le}\cdots{\le}M_{r_2}^{(2)}{<_{1-c^{'}_2}}m_2\\
\vdots\\
m_{q-1}{<_{c^{'}_{q-1}}}M_{1}^{(q)}{\le}\cdots{\le}M_{r_q}^{(q)}{<_{1-c^{'}_q}}m_q
\end{subarray}}
\frac{(\beta)_{m_1}}{{m_1}!}
\frac{{m_q}!(m_q+\alpha)}{(\beta)_{m_q+1}}\\
&\times\frac{1}{(m_1+\beta)^{c^{'}_{1}r_1}(m_1+\alpha)^{k^{'}_1}}
\left\{\prod_{i=2}^{q}\left(\prod_{j=1}^{r_i}\frac{1}{M_{j}^{(i)}+\beta}\right)
\frac{1}{(m_i+\alpha)^{k^{'}_i}}\right\}
\end{aligned}
\end{equation}
for $r\ge0$, $m_1,\ldots,m_q\in\mathbb{Z}$ such that $0{\le}m_1{<_{c^{'}_1}}\cdots{<_{c^{'}_{q-1}}}m_q$ 
and $c^{'}_q=1$. Differentiating the right-hand side of (9) $r$ times with respect to $\beta$ 
and using (14), we obtain the right-hand side of (6). This completes the proof of Theorem 1.1 (i). 
\subsection{Proof of Theorem 1.1 (ii) and a related theorem}
For brevity, we put $\mathbf{y}_{m}^{(i)}:=\sum_{j=1}^{m}y^{(i)}_{j}$ 
($i,m\ge1$, $y^{(i)}_{j}\ge0$): we regard $\mathbf{y}_{0}^{(i)}$ as $0$. 
For any $v=\prod_{i=1}^{p}z_{c_{i-1}}(k_i)\in{B^0}$ and its dual 
$\tau(v)=\prod_{i=1}^{q}z_{c^{'}_{i-1}}(k^{'}_i)$, we define the following two monomials: 
\begin{equation}
v_{\mathbf{y}}
:=\left\{\prod_{i=1}^{p-1}z_{c_{i-1}}(k_i+\mathbf{y}_{k_i-c_{i}}^{(i)})\right\}
z_{c_{p-1}}(k_p+\mathbf{y}_{k_p-2}^{(p)}),
\end{equation}
\begin{equation}
\begin{aligned}
v^{'}_{(\{l_i\}_{i=1}^{q-1})}
&:=x_{1}x_{-1}^{k^{'}_1-1}\left\{\prod_{i=2}^{q}x_{c^{'}_{i-1}}x_{1}^{l_{i-1}}x_{-1}^{k^{'}_{i}-1}\right\}\\
&=z_{1}(k^{'}_1)\left\{\prod_{i=2}^{q}z_{c^{'}_{i-1}}(1)z_{1}(1)^{l_{i-1}-1}z_{1}(k^{'}_{i})\right\},
\end{aligned}
\end{equation}
where $\mathbf{y}:=(\{\mathbf{y}_{k_i-c_{i}}^{(i)}\}_{i=1}^{p-1}, \mathbf{y}_{k_p-2}^{(p)})$ and $l_i\ge0$ ($i=1,\ldots,q-1$). 
If $l_{i-1}=0$ in (16), the factor becomes $x_{c^{'}_{i-1}}x_{1}^{l_{i-1}}x_{-1}^{k^{'}_{i}-1}=x_{c^{'}_{i-1}}x_{-1}^{k^{'}_{i}-1}
=z_{c^{'}_{i-1}}(k^{'}_{i})$; therefore we regard $z_{c^{'}_{i-1}}(1)z_{1}(1)^{-1}z_{1}(k^{'}_{i})$ as $z_{c^{'}_{i-1}}(k^{'}_{i})$. 
\par 
Theorem 1.1 (ii) can be proved in the same way as in \cite[Section 2]{i2009}. 
We first prove the following identity:
\begin{proposition}
Let $v\in{B^0}$. Then 
\begin{equation}
\begin{aligned}
&
\sum_{l=0}^{r}
\sum_{\begin{subarray}{c}
\sum_{i=1}^{p-1}\mathbf{y}_{k_i-c_{i}}^{(i)}+\mathbf{y}^{(p)}_{k_p-2}=l\\
y_{j}^{(i)}\ge0\end{subarray}}
Z(\sigma_{r-l}(v_{\mathbf{y}});\alpha)\\
=
&\sum_{l=0}^{r}\sum_{\begin{subarray}{c}l_1+\cdots+l_{q-1}=l\\
l_i\ge0\end{subarray}}
Z(\sigma_{r-l}(v^{'}_{(\{l_i\}_{i=1}^{q-1})});\alpha)
\end{aligned}
\end{equation}
for all $r\ge0$, $\alpha\in\mathbb{C}$ with $\mathrm{Re}(\alpha)>0$.
\end{proposition}
\begin{proof}
The proof is similar to that of \cite[Lemma 2.5]{i2009}. 
Let $v=\prod_{i=1}^{p}z_{c_{i-1}}(k_i)\in{B^0}$, and let 
$\tau(v)=\prod_{i=1}^{q}z_{c^{'}_{i-1}}(k^{'}_i)$ be its dual. 
 Here we note that the identities 
\begin{equation}
\begin{aligned}
&\frac{(-1)^r}{r!}\frac{\mathrm{d}^r}{\mathrm{d}\beta^r}
\left(\left\{\prod_{i=1}^{p-1}\frac{1}{(m_i+\beta)^{k_i}}\right\}\frac{1}{(m_p+\beta)^{k_p-1}}\right)\\
=
&
\sum_{\begin{subarray}{c}r_1+\cdots+r_p=r\\
r_i\ge0\end{subarray}}
\left(
\prod_{i=1}^{p-1}\frac{\binom{k_i+r_i-1}{r_i}}{(m_i+\beta)^{k_i+r_i}}\right)\frac{\binom{k_p+r_p-2}{r_p}}{(m_p+\beta)^{k_p+r_p-1}}\\
=
&
\sum_{l=0}^{r}
\sum_{\begin{subarray}{c}
\sum_{i=1}^{p-1}\mathbf{y}_{k_i-c_{i}}^{(i)}\\+\mathbf{y}^{(p)}_{k_p-2}=l\\
y_{j}^{(i)}\ge0\end{subarray}}
\sum_{\begin{subarray}{c}\sum_{i=1}^{p-1}c_{i}r_i\\+r_{p}=r-l\\
c_{i}r_i, r_p\ge0\end{subarray}}
\left(\prod_{i=1}^{p-1}\frac{1}{(m_i+\beta)^{k_i+\mathbf{y}_{k_i-c_{i}}^{(i)}+c_{i}r_i}}\right)\\
&\times\frac{1}{(m_p+\beta)^{k_p+\mathbf{y}^{(p)}_{k_p-2}+r_p-1}}
\end{aligned}
\end{equation}
hold for $r\ge0$ and $m_1,\ldots,m_p\in\mathbb{Z}$ such that 
$0{\le}m_1{<_{c_1}}\cdots{<_{c_{p-1}}}m_p$. Therefore, using (18), we have 
\begin{equation}
\begin{aligned}
&\frac{(-1)^r}{r!}\frac{\partial^r}{\partial\beta^r}
\left(\frac{(\alpha)_{m_1}}{m_1!}\frac{m_p!}{(\alpha)_{m_p+1}}\left\{\prod_{i=1}^{p-1}\frac{1}{(m_i+\beta)^{k_i}}\right\}\frac{1}{(m_p+\beta)^{k_p-1}}\right){\Biggl|}_{\beta=\alpha}\\
=
&\sum_{l=0}^{r}
\sum_{\begin{subarray}{c}
\sum_{i=1}^{p-1}\mathbf{y}_{k_i-c_{i}}^{(i)}\\+\mathbf{y}^{(p)}_{k_p-2}=l\\
y_{j}^{(i)}\ge0\end{subarray}}
\sum_{\begin{subarray}{c}\sum_{i=1}^{p-1}c_{i}r_i\\+r_{p}=r-l\\
c_{i}r_i, r_p\ge0\end{subarray}}
\frac{(\alpha)_{m_1}}{m_1!}\frac{m_p!}{(\alpha)_{m_p}}
\left(\prod_{i=1}^{p-1}\frac{1}{(m_i+\alpha)^{k_i+\mathbf{y}_{k_i-c_{i}}^{(i)}+c_{i}r_i}}\right)\\
&\times\frac{1}{(m_p+\alpha)^{k_p+\mathbf{y}^{(p)}_{k_p-2}+r_p}}
\end{aligned}
\end{equation}
for $r\ge0$ and $m_1,\ldots,m_p\in\mathbb{Z}$ such that 
$0{\le}m_1{<_{c_1}}\cdots{<_{c_{p-1}}}m_p$. Differentiating the left-hand side of (9) $r$ times with respect to $\beta$ at $\beta=\alpha$ ($\alpha\in\mathbb{C}$ with $\mathrm{Re}(\alpha)>0$) and using (19), we obtain the left-hand 
side of (17). To obtain the right-hand side of (17), we use the same calculus as 
in the proof of Theorem 1.1 (i). 
From the definition of $c_i$ ($=0,1$), we have $(\beta)_{m_i+c^{'}_{i}}=(\beta)_{m_i}(m_i+\beta)^{c^{'}_{i}}$. Using this, we have the expression 
\begin{equation}
\frac{(\beta)_{m_1}}{(\beta)_{m_q+1}}
=
\left(\prod_{i=1}^{q-1}\frac{(\beta)_{m_i+c^{'}_{i}}}{(\beta)_{m_{i+1}}}\right)
\left(\prod_{i=1}^{q-1}\frac{1}{(m_i+\beta)^{c^{'}_{i}}}\right)
\frac{1}{m_q+\beta},
\end{equation}
where $m_1,\ldots,m_q\in\mathbb{Z}$ such that $0{\le}m_1{<_{c^{'}_1}}\cdots{<_{c^{'}_{q-1}}}m_q$. The derivatives of the factors on the 
right-hand side of (20) can be calculated as follows: 
\begin{equation}
\begin{aligned}
&\frac{(-1)^{s_i}}{{s_i}!}\frac{\mathrm{d}^{s_i}}{\mathrm{d}\beta^{s_i}}
\left(\frac{(\beta)_{m_i+c^{'}_{i}}}{(\beta)_{m_{i+1}}}\right)\\
=
&\frac{(\beta)_{m_i+c^{'}_{i}}}{(\beta)_{m_{i+1}}}
\sum_{\begin{subarray}{c}m_i+c^{'}_{i}{\le}M_1^{(i)}{\le}\cdots{\le}M_{s_i}^{(i)}<m_{i+1}\end{subarray}}
\prod_{j=1}^{s_i}\frac{1}{M^{(i)}_{j}+\beta}\\
=
&\frac{(\beta)_{m_i+c^{'}_{i}}}{(\beta)_{m_{i+1}}}
\sum_{\begin{subarray}{c}m_i{<_{c^{'}_i}}M_1^{(i)}{\le}\cdots{\le}M_{s_i}^{(i)}<m_{i+1}\end{subarray}}
\prod_{j=1}^{s_i}\frac{1}{M^{(i)}_{j}+\beta}\\
=
&\frac{(\beta)_{m_i+c^{'}_{i}}}{(\beta)_{m_{i+1}}}
\sum_{l_i=0}^{s_i}\sum_{\begin{subarray}{c}\sum_{j=1}^{l_i}y_j^{(i)}=s_i-l_i\\
y_j^{(i)}\ge0\end{subarray}}
\sum_{\begin{subarray}{c}m_i{<_{c^{'}_i}}M_1^{(i)}<\cdots<M_{l_i}^{(i)}<m_{i+1}\end{subarray}}
\prod_{j=1}^{l_i}\frac{1}{(M^{(i)}_{j}+\beta)^{y_{j}^{(i)}+1}}
\end{aligned}
\end{equation}
for $s_i,m_i,m_{i+1}\ge0$ such that $m_i{<_{c^{'}_i}}m_{i+1}$ ($i=1,\ldots,q-1$), 
where we regard 
$\sum_{\begin{subarray}{c}0=s_i\\y_j^{(i)}\ge0\end{subarray}}$ as 1 if $s_i=0$ and as 0 otherwise. 
(From the definitions of the symbols $c_i$ and $<_{c_i}$, 
the inequality $m_i+c^{'}_{i}{\le}M_1^{(i)}$ of (21) can be rewritten as $m_i{<_{c^{'}_i}}M_1^{(i)}$.) 
Using (13), (20) and (21), we have 
\begin{equation*}
\begin{aligned}
&\frac{(-1)^r}{r!}\frac{\mathrm{d}^r}{\mathrm{d}\beta^r}\left(\frac{(\beta)_{m_1}}{(\beta)_{m_q+1}}\right)\\
=
&\sum_{\begin{subarray}{c}\sum_{i=1}^{q}r_i+\sum_{i=1}^{q-1}s_i=r\\
r_i,s_i\ge0\end{subarray}}
\left(\prod_{i=1}^{q-1}\frac{\binom{r_i+c^{'}_i-1}{r_i}}{(m_i+\beta)^{r_i+c^{'}_{i}}}\right)
\frac{1}{(m_q+\beta)^{r_q+1}}\\
&
\times\prod_{i=1}^{q-1}
\left(\frac{(\beta)_{m_i+c^{'}_{i}}}{(\beta)_{m_{i+1}}}
\sum_{l_i=0}^{s_i}\sum_{\begin{subarray}{c}\sum_{j=1}^{l_i}y_j^{(i)}=s_i-l_i\\
y_j^{(i)}\ge0\end{subarray}}
\sum_{\begin{subarray}{c}m_i{<_{c^{'}_i}}M_1^{(i)}<\cdots<M_{l_i}^{(i)}<m_{i+1}\end{subarray}}
\prod_{j=1}^{l_i}\frac{1}{(M^{(i)}_{j}+\beta)^{y_{j}^{(i)}+1}}\right)\\
=
&\sum_{\begin{subarray}{c}\sum_{i=1}^{q-1}c^{'}_ir_i+r_q\\
+\sum_{i=1}^{q-1}s_i=r\\
c^{'}_{i}r_i, r_q,s_i\ge0\end{subarray}}
\frac{(\beta)_{m_1}}{(\beta)_{m_q}}
\left(\prod_{i=1}^{q-1}\frac{1}{(m_i+\beta)^{c^{'}_{i}r_i}}\right)
\frac{1}{(m_q+\beta)^{r_q+1}}\\
&
\times\prod_{i=1}^{q-1}
\left(\sum_{l_i=0}^{s_i}\sum_{\begin{subarray}{c}\sum_{j=1}^{l_i}y_j^{(i)}=s_i-l_i\\
y_j^{(i)}\ge0\end{subarray}}
\sum_{\begin{subarray}{c}m_i{<_{c^{'}_i}}M_1^{(i)}<\cdots<M_{l_i}^{(i)}<m_{i+1}\end{subarray}}
\prod_{j=1}^{l_i}\frac{1}{(M^{(i)}_{j}+\beta)^{y_{j}^{(i)}+1}}\right)\\
=
&\sum_{l=0}^{r}
\sum_{\begin{subarray}{c}l_1+\cdots+l_{q-1}=l\\
l_i\ge0\end{subarray}}
\sum_{\begin{subarray}{c}\sum_{i=1}^{q-1}c^{'}_ir_i+r_q\\
+\sum_{i=1}^{q-1}\sum_{j=1}^{l_i}y_j^{(i)}=r-l\\
c^{'}_{i}r_i, r_q,y_j^{(i)}\ge0\end{subarray}}
\frac{(\beta)_{m_1}}{(\beta)_{m_q}}
\left(\prod_{i=1}^{q-1}\frac{1}{(m_i+\beta)^{c^{'}_{i}r_i}}\right)
\frac{1}{(m_q+\beta)^{r_q+1}}\\
&\times
\prod_{i=1}^{q-1}\left(\sum_{\begin{subarray}{c}m_i{<_{c^{'}_i}}M_1^{(i)}<\cdots<M_{l_i}^{(i)}<m_{i+1}\end{subarray}}
\prod_{j=1}^{l_i}\frac{1}{(M^{(i)}_{j}+\beta)^{y_{j}^{(i)}+1}}\right)
\end{aligned}
\end{equation*}
for $r\ge0$. Therefore, using this result, we have 
\begin{equation}
\begin{aligned}
&\frac{(-1)^r}{r!}\frac{\partial^r}{\partial\beta^r}\left(\frac{(\beta)_{m_1}}{m_1!}\frac{m_q!}{(\beta)_{m_q+1}}
\left\{\prod_{i=1}^{q-1}\frac{1}{(m_i+\alpha)^{k^{'}_i}}\right\}\frac{1}{(m_q+\alpha)^{k^{'}_q-1}}\right){\Biggl|}_{\beta=\alpha}\\
=
&\sum_{l=0}^{r}
\sum_{\begin{subarray}{c}l_1+\cdots+l_{q-1}=l\\
l_i\ge0\end{subarray}}
\sum_{\begin{subarray}{c}\sum_{i=1}^{q-1}c^{'}_ir_i+r_q\\
+\sum_{i=1}^{q-1}\sum_{j=1}^{l_i}y_j^{(i)}=r-l\\
c^{'}_{i}r_i, r_q,y_j^{(i)}\ge0\end{subarray}}
\sum_{\begin{subarray}{c}m_1{<_{c^{'}_1}}M_1^{(1)}<\cdots<M_{l_1}^{(1)}<m_2\\
\vdots\\
m_{q-1}{<_{c^{'}_{q-1}}}M_1^{(q-1)}<\cdots<M_{l_{q-1}}^{(q-1)}<m_q
\end{subarray}}
\frac{(\alpha)_{m_1}}{m_1!}\frac{m_q!}{(\alpha)_{m_q}}\\
&\times
\left\{\prod_{i=1}^{q-1}\frac{1}{(m_i+\alpha)^{k^{'}_i+c^{'}_{i}r_i}}
\left(\prod_{j=1}^{l_i}\frac{1}{(M^{(i)}_{j}+\alpha)^{y_{j}^{(i)}+1}}\right)\right\}
\frac{1}{(m_q+\alpha)^{k^{'}_q+r_q}}
\end{aligned}
\end{equation}
for $r\ge0$ and $m_1,\ldots,m_q\in\mathbb{Z}$ such that 
$0{\le}m_1{<_{c^{'}_1}}\cdots{<_{c^{'}_{q-1}}}m_q$. 
Differentiating the right-hand side of (9) $r$ times with respect to $\beta$ at $\beta=\alpha$ ($\alpha\in\mathbb{C}$ with $\mathrm{Re}(\alpha)>0$) and using (22), 
we obtain the right-hand side of (17). This completes the proof of (17).
\end{proof}
\begin{proposition}
Theorem $1.1$ $(ii)$ and Proposition $2.4$ are equivalent.
\end{proposition}
\begin{proof}
Let $v\in{B^0}$, and let $v_{\mathbf{y}}$ be the monomial defined by (15). 
From the definition of the dual map $\tau$, it can be seen that the dual 
$\tau(v_{\mathbf{y}})$ takes the form
\begin{equation}
x_{1}x_{-1}^{k^{'}_1-1}\left\{\prod_{i=2}^{q}x_{c^{'}_{i-1}}x_{1}^{l_{i-1}}x_{-1}^{k^{'}_{i}-1}\right\}
=z_{1}(k^{'}_1)\left\{\prod_{i=2}^{q}z_{c^{'}_{i-1}}(1)z_{1}(1)^{l_{i-1}-1}z_{1}(k^{'}_{i})\right\}
\end{equation}
with
\begin{equation*}
l_1+\cdots+l_{q-1}=\sum_{i=1}^{p-1}\mathbf{y}_{k_i-c_{i}}^{(i)}+\mathbf{y}^{(p)}_{k_p-2},
\end{equation*} 
where $l_i\ge0$ ($i=1,\ldots,q-1$) and the parameters $q$, $c^{'}_{i-1}$ and $k^{'}_{i}$ are those of $\tau(v)=\prod_{i=1}^{q}z_{c^{'}_{i-1}}(k^{'}_i)$. The monomial (23) is just $v^{'}_{(\{l_i\}_{i=1}^{q-1})}$, the monomial defined by (16). 
The equivalence can be proved by using these facts on monomials and the same argument as in \cite[Proof of Proposition 2.7]{i2009}.
\end{proof}
We can prove also the following theorem, which shows that the same equivalence as in \cite[Theorem 1.2]{i2009} holds also for the extended multiple series:
\begin{theorem}
The following assertion $(A)$ is equivalent to Theorem $1.1$ $(ii):$
\par
$(A)$ Let $v\in{B^0}$, and let $\tau(v)$ be its dual. Then the identity $(7)$ 
holds for all ``even'' integers $r\ge0$ and $\alpha\in\mathbb{C}$ with $\mathrm{Re}(\alpha)>0$.
\end{theorem}
\begin{proof}
This can be proved in the same way as in \cite[Section 3]{i2009}.
\end{proof}
Applying Theorem 2.6 to the case $v=\prod_{i=1}^{p}z_{1}(k_i)$, 
we have Theorem 1.2 of \cite{i2009}.
\begin{example}
We give some examples of Theorem 1.1:
\par
(i) We put $v_0:=z_1(k_1)\left\{\prod_{i=2}^{p}z_{0}(k_i)\right\}$ 
($p,k_i\ge1$ ($i=1,\ldots,p-1$), $k_p\ge2$). Then we have 
\begin{equation*}
\tau(v_0)=\left\{\prod_{i=p}^{2}x_{1}^{k_i-1}x_{0}\right\}x_{1}^{k_1-1}x_{-1}
=
\left\{\prod_{i=1}^{q-1}
z_{c^{'}_{i-1}}(1)\right\}z_{c^{'}_{q-1}}(2),
\end{equation*}
where $q\ge1$, $c^{'}_0=1$, $c^{'}_i\in\{0,1\}$ ($i=1,\ldots,q-1$). Because $\tau^2(v_0)=v_0$, the dual of $\tau(v_0)$ becomes $v_0$. 
For these monomials, the identity (6) gives the following two different relations among (2), 
which come from the symmetry on $\alpha$ and $\beta$ of (9) (compare (14) with (19)): by taking $v=v_0$ in (6), 
\begin{equation}
\begin{aligned}
&\sum_{\begin{subarray}{c}r_1+\cdots+r_p=r\\
r_i\ge0\end{subarray}}
\left\{\prod_{i=1}^{p-1}\binom{k_{i}+r_{i}-1}{r_{i}}\right\}
\binom{k_{p}+r_p-2}{r_{p}}\\
&{\times}Z\left(z_1(k_1+r_1)\left\{\prod_{i=2}^{p}z_{0}(k_i+r_i)\right\}; 
(\alpha,\beta)\right)\\
=
&\sum_{\begin{subarray}{c}c^{'}_{1}r_1+\sum_{i=2}^{q}r_i=r\\
c^{'}_{1}r_1, r_i\ge0\end{subarray}}
Z^{*}_{(c^{'}_1r_1,\{r_i\}_{i=2}^{q})}
\left(\left\{\prod_{i=1}^{q-1}
z_{c^{'}_{i-1}}(1)\right\}z_{c^{'}_{q-1}}(2);(\beta,\alpha)\right)
\end{aligned}
\end{equation}
and, by taking $v=\tau(v_0)$ in (6), 
\begin{equation}
\begin{aligned}
&\sum_{\begin{subarray}{c}r_1+\cdots+r_q=r\\
r_i\ge0\end{subarray}}
Z\left(\left\{\prod_{i=1}^{q-1}z_{c^{'}_{i-1}}(1+r_i)\right\}z_{c^{'}_{q-1}}(2+r_q);
(\alpha,\beta)\right)\\
=
&\sum_{\begin{subarray}{c}\varepsilon(p)r_1+\sum_{i=2}^{p}r_i=r\\
\varepsilon(p)r_1, r_i\ge0\end{subarray}}
Z^{*}_{(\varepsilon(p)r_1,\{r_i\}_{i=2}^{p})}
\left(z_1(k_1)\left\{\prod_{i=2}^{p}z_{0}(k_i)\right\};
(\beta,\alpha)\right)
\end{aligned}
\end{equation}
for all $r\ge0$, $\alpha,\beta\in\mathbb{C}$ with 
$\mathrm{Re}(\alpha), \mathrm{Re}(\beta)>0$, where $\varepsilon(p)=1$ 
if $p=1$ and $\varepsilon(p)=0$ otherwise. 
If $p=1$, the right-hand side of (25) becomes 
$Z^{*}_{(r)}(z_1(k_1);(\beta,\alpha))=\sum_{m=0}^{\infty}(m+\beta)^{-r-1}(m+\alpha)^{-k_1+1}$; therefore 
the case $p=1$ of (25) is the sum formula proved in 
\cite{ig2007} (see also \cite[Remark 2.4]{i2009} and \cite[\textbf{(R2)}]{i2018}). 
We note that, for the monomials $v_0$ and $\tau(v_0)$, 
the identity (7) gives only the following one relation among (2) with $\alpha=\beta$:
\begin{equation}
\begin{aligned}
&Z\left(z_1(k_1)\left\{\prod_{i=2}^{p-1}z_{0}(k_i)\right\}
z_{0}(k_p+r);\alpha\right)\\
=
&\sum_{\begin{subarray}{c}\sum_{i=1}^{q-1}c^{'}_{i}r_i+r_{q}=r\\
c^{'}_{i}r_i, r_q\ge0\end{subarray}}
Z\left(\left\{\prod_{i=1}^{q-1}z_{c^{'}_{i-1}}(1+c^{'}_ir_i)\right\}z_{c^{'}_{q-1}}(2+r_q);\alpha\right)
\end{aligned}
\end{equation}
for all $r\ge0$, $\alpha\in\mathbb{C}$ with $\mathrm{Re}(\alpha)>0$.
\par
(ii) Since the dual of $\prod_{i=1}^{p}z_{1}(k_i)$ is 
$\prod_{i=1}^{q}z_{1}(k^{'}_i)$, the identity (6) with $v=\prod_{i=1}^{p}z_{1}(k_i)$ 
becomes 
\begin{equation*}
\begin{aligned}
&\sum_{\begin{subarray}{c}r_{1}+\cdots+r_{p}=r\\
r_i\ge0\end{subarray}}
\left\{\prod_{i=1}^{p-1}\binom{k_{i}+r_{i}-1}{r_{i}}\right\}
\binom{k_{p}+r_{p}-2}{r_{p}}
Z\left(\prod_{i=1}^{p}z_{1}(k_i+r_i);(\alpha,\beta)\right)\\
=
&\sum_{\begin{subarray}{c}r_{1}+\cdots+r_q=r\\
r_i\ge0\end{subarray}}
Z^{*}_{(\{r_i\}_{i=1}^{q})}
\left(\prod_{i=1}^{q}z_{1}(k^{'}_i+r_i);(\beta,\alpha)\right)
\end{aligned}
\end{equation*}
for all $r\ge0$, $\alpha,\beta\in\mathbb{C}$ with $\mathrm{Re}(\alpha), \mathrm{Re}(\beta)>0$.
\end{example}
\section{Duality of multiple Hurwitz zeta values}
In the present section, we apply our method used in Section 2 to 
deriving duality relations for multiple Hurwitz zeta values. 
Our result in the present section is also formulated in the same way as in Introduction. 
We define the evaluation map $\zeta=\zeta_{\alpha}:B^0 \rightarrow \mathbb{C}$ 
by $\zeta(1;\alpha)=1$ and 
\begin{equation}
\begin{aligned}
\zeta(z_1(k_1)z_{c_1}(k_2){\cdots}z_{c_{p-1}}(k_p);\alpha)
=\sum_{\begin{subarray}{c}0{\le}m_1<_{c_1}\cdots<_{c_{p-1}}m_p<\infty\end{subarray}}
\prod_{i=1}^{p}\frac{1}{(m_i+\alpha)^{k_i}},
\end{aligned}
\end{equation}
where $p\ge1$ and $\alpha\in\mathbb{C}\setminus\mathbb{Z}_{\le0}$. 
This map can be extended to $\mathbb{Q}$-linear maps onto the whole space $V^0$. 
We call the multiple series (27) the multiple Hurwitz zeta value (MHZV for short). 
In \cite{ig2007} and \cite{i2018}, we studied relations for MHZVs in 
some different ways. Our results were described as relations between MHZVs and the multiples series 
\begin{equation}
\begin{aligned}
\sum_{\begin{subarray}{c}0{\le}m_1<_{c_1}\cdots<_{c_{p-1}}m_p<\infty\end{subarray}}
z^{m_p}\frac{{m_p}!}{(\alpha)_{m_p}}
\left\{\prod_{i=1}^{p}\frac{1}{(m_i+\alpha)^{a_i}(m_i+1)^{b_i}}\right\},
\end{aligned}
\end{equation}
where $z\in\{-1,1\}$; $p\ge1$; $a_{i}, b_{i}\in\mathbb{Z}$ such that $a_{i}+b_{i}\ge1$ ($i=1,\ldots,p-1$), $a_{p}+b_{p}\ge2$; 
$\alpha\in\mathbb{C}$ with $\mathrm{Re}(\alpha)>0$; $(a)_m$ is the Pochhammer symbol. (For related works, see Remark 3.3 below and also 
Note at the end of this section.) 
Our result in the present section, a duality of MHZVs, is also described as such a relation. 
To formulate it, we introduce the evaluation map 
$H^{*}_{(\{r_i\}_{i=1}^{q})}=H^{*}_{(\{r_i\}_{i=1}^{q}), \alpha}:B^0 \rightarrow \mathbb{C}$ defined by $H^{*}_{(\{r_i\}_{i=1}^{q})}(1;\alpha)=1$ 
and 
\begin{equation}
\begin{aligned}
&H^{*}_{(\{r_i\}_{i=1}^{q})}(z_1(k_1)z_{c_1}(k_2){\cdots}z_{c_{q-1}}(k_q);\alpha)\\
=
&\sum_{\begin{subarray}{c}0{\le}M_{1}^{(1)}{\le}\cdots{\le}M_{r_1}^{(1)}<_{1-c_1}m_1\\
\vdots\\
m_{i-1}<_{c_{i-1}}M_{1}^{(i)}{\le}\cdots{\le}M_{r_i}^{(i)}<_{1-c_i}m_i\\
\vdots\\
m_{q-1}<_{c_{q-1}}M_{1}^{(q)}{\le}\cdots{\le}M_{r_q}^{(q)}<_{1-c_q}m_q<\infty
\end{subarray}}
\frac{(m_q+1)!}{(\alpha)_{m_q+1}}
\left\{\prod_{i=1}^{q}\left(\prod_{j=1}^{r_i}\frac{1}{M_{j}^{(i)}+\alpha}\right)
\frac{1}{(m_i+1)^{k_i}}\right\},
\end{aligned}
\end{equation}
where $q\ge1$, $r_i\ge0$ ($i=1,\ldots,q$), $c_q=1$, 
$\alpha\in\mathbb{C}$ with $\mathrm{Re}(\alpha)>0$. 
If $r_i=0$, we regard the inequalities $m_{i-1}<_{c_{i-1}}M_{1}^{(i)}{\le}\cdots{\le}M_{r_i}^{(i)}<_{1-c_i}m_i$ of (29) as $m_{i-1}<_{c_{i-1}}m_i$. 
This map can be extended to $\mathbb{Q}$-linear maps onto the whole space $V^0$. 
We use also the map $\sigma_r^{b,2}:B^0 \rightarrow V^0$ defined by $\sigma_r^{b,2}(1)=1$ and 
\begin{equation*}
\begin{aligned}
&\sigma_r^{b,2}(z_1(k_1)z_{c_1}(k_2){\cdots}z_{c_{p-1}}(k_p))\\
=
&\sum_{\begin{subarray}{c}r_1+\cdots+r_p=r\\
r_i\ge0\end{subarray}}
\left\{\prod_{i=1}^{p}\binom{k_{i}+r_{i}-1}{r_{i}}\right\}
\prod_{i=1}^{p}z_{c_{i-1}}(k_i+r_i),
\end{aligned}
\end{equation*}
where $r\ge0$. This can be extended to a $\mathbb{Q}$-linear map from the whole space $V^0$ to itself. Using the same method as in Section 2, we can prove the following new relation between MHZVs and (28) with $z=1$, 
which yields numerous relations: 
\begin{theorem}
Let $v\in{B^0}$, and let $\tau(v)$ be its dual. Then 
\begin{equation}
\zeta(\sigma_r^{b,2}(v);\alpha)
=
\sum_{\begin{subarray}{c}r_1+\cdots+r_{q}=r\\
r_i\ge0\end{subarray}}
H^{*}_{(\{r_i\}_{i=1}^{q})}(\tau(v);\alpha)
\end{equation}
for all $r\ge0$, $\alpha\in\mathbb{C}$ with $\mathrm{Re}(\alpha)>0$.
\end{theorem}
\begin{proof}
Let $v=\prod_{i=1}^{p}z_{c_{i-1}}(k_i)\in{B^0}$, and let 
$\tau(v)=\prod_{i=1}^{q}z_{c^{'}_{i-1}}(k^{'}_i)$ be its dual. Using the same way as in the proof of Lemma 2.1 with $\alpha=1$, we have the following 
iterated integral representation of (27):
\begin{equation}
\begin{aligned}
&\zeta(x_{1}x_{e_1}{\cdots}x_{e_{n-1}}x_{-1};\alpha)\\
=
&\idotsint\displaylimits_{\begin{subarray}{c}
0<t_0<\cdots<t_n<1
\end{subarray}}
t_0^{\alpha-1}\omega_1(t_0)
\left\{\prod_{i=1}^{n-1}\omega_{e_i}(t_i)\right\}
\omega_{-1}(t_n)\mathrm{d}t_{0}\cdots\mathrm{d}t_{n}
\end{aligned}
\end{equation}
for $\alpha\in\mathbb{C}$ with $\mathrm{Re}(\alpha)>0$, 
where $n\ge1$ and $e_i\in\{-1,0,1\}$ $(i=1,\ldots,n-1)$. 
Further, making the change of variables $t_i=1-u_{n-i}$ 
($i=0,1,\ldots,n$) to the above iterated integral, we have the duality formula
\begin{equation}
\begin{aligned}
\zeta(v;\alpha)
=
&\sum_{\begin{subarray}{c}0{\le}m_1<_{c^{'}_1}\cdots<_{c^{'}_{q-1}}m_q<\infty\end{subarray}}
\frac{(m_q+1)!}{(\alpha)_{m_q+1}}
\left\{\prod_{i=1}^{q}\frac{1}{(m_i+1)^{k^{'}_i}}\right\}\\
=
&H^{*}_{(\{0\}_{i=1}^q)}(\tau(v);\alpha)
\end{aligned}
\end{equation}
for $\alpha\in\mathbb{C}$ with $\mathrm{Re}(\alpha)>0$. 
The left-hand side of (30) can be obtained by differentiating that of (32) $r$ times. 
The right-hand side of (30) can also be obtained in the same way as in 
the proof of Theorem 1.1 (i). Indeed, dividing both sides of (10) by $(\beta)_{m_1}$, we have the expression 
\begin{equation*}
\frac{1}{(\alpha)_{m_q+1}}
=
\frac{1}{(\alpha)_{m_1+c^{'}_1}}\left(\prod_{i=2}^{q}\frac{(\alpha)_{m_{i-1}+c^{'}_{i-1}}}
{(\alpha)_{m_i+c^{'}_i}}\right),
\end{equation*}
where $m_1,\ldots,m_q\in\mathbb{Z}$ such that $0{\le}m_1{<_{c^{'}_1}}\cdots{<_{c^{'}_{q-1}}}m_q$ and $c^{'}_q=1$. 
Using this and a calculus similar to in the proof of (12), we have 
\begin{equation}
\begin{aligned}
&\frac{(-1)^r}{r!}\frac{\mathrm{d}^r}{\mathrm{d}\alpha^r}\left(\frac{(m_q+1)!}{(\alpha)_{m_q+1}}
\left\{\prod_{i=1}^{q}\frac{1}{(m_i+1)^{k^{'}_i}}\right\}\right)\\
=
&\sum_{\begin{subarray}{c}r_1+\cdots+r_q=r\\
r_i\ge0\end{subarray}}
\left(\frac{(m_q+1)!}{(\alpha)_{m_1+c^{'}_1}}
\sum_{\begin{subarray}{c}0{\le}M_1^{(1)}{\le}\cdots{\le}M^{(1)}_{r_{1}}<m_{1}+c^{'}_1\end{subarray}}
\prod_{j=1}^{r_1}\frac{1}{M^{(1)}_{j}+\alpha}\right)\\
&\times
\left(\prod_{i=2}^{q}\frac{(\alpha)_{m_{i-1}+c^{'}_{i-1}}}{(\alpha)_{m_i+c^{'}_i}}
\sum_{\begin{subarray}{c}m_{i-1}+c^{'}_{i-1}{\le}M_1^{(i)}{\le}\cdots{\le}M^{(i)}_{r_{i}}<m_{i}+c^{'}_i\end{subarray}}
\prod_{j=1}^{r_i}\frac{1}{M^{(i)}_{j}+\alpha}\right)\\
&\times\left\{\prod_{i=1}^{q}\frac{1}{(m_i+1)^{k^{'}_i}}\right\}\\
=
&
\sum_{\begin{subarray}{c}r_1+\cdots+r_q=r\\
r_i\ge0\end{subarray}}
\frac{(m_q+1)!}{(\alpha)_{m_q+1}}
\left(\sum_{\begin{subarray}{c}0{\le}M_1^{(1)}{\le}\cdots{\le}M^{(1)}_{r_{1}}{<_{1-c^{'}_1}}m_1\end{subarray}}
\prod_{j=1}^{r_1}\frac{1}{M^{(1)}_{j}+\alpha}\right)\\
&\times
\left\{\prod_{i=2}^{q}\left(
\sum_{\begin{subarray}{c}m_{i-1}{<_{c^{'}_{i-1}}}M_1^{(i)}{\le}\cdots{\le}M^{(i)}_{r_{i}}{<_{1-c^{'}_i}}m_i\end{subarray}}
\prod_{j=1}^{r_i}\frac{1}{M^{(i)}_{j}+\alpha}\right)\right\}\left\{\prod_{i=1}^{q}\frac{1}{(m_i+1)^{k^{'}_i}}\right\}\\
=
&
\sum_{\begin{subarray}{c}r_1+\cdots+r_q=r\\
r_i\ge0\end{subarray}}
\sum_{\begin{subarray}{c}0{\le}M_{1}^{(1)}{\le}\cdots{\le}M_{r_1}^{(1)}{<_{1-c^{'}_1}}m_1\\
\vdots\\
m_{i-1}{<_{c^{'}_{i-1}}}M_1^{(i)}{\le}\cdots{\le}M^{(i)}_{r_{i}}{<_{1-c^{'}_i}}m_i\\
\vdots\\
m_{q-1}{<_{c^{'}_{q-1}}}M_{1}^{(q)}{\le}\cdots{\le}M_{r_q}^{(q)}{<_{1-c^{'}_q}}m_q
\end{subarray}}
\frac{(m_q+1)!}{(\alpha)_{m_q+1}}
\left\{\prod_{i=1}^{q}\left(\prod_{j=1}^{r_i}\frac{1}{M_{j}^{(i)}+\alpha}\right)
\frac{1}{(m_i+1)^{k^{'}_i}}\right\}
\end{aligned}
\end{equation}
for $r\ge0$, $m_1,\ldots,m_q\in\mathbb{Z}$ such that $0{\le}m_1{<_{c^{'}_1}}\cdots{<_{c^{'}_q-1}}m_q$ 
and $c^{'}_q=1$. Therefore, differentiating the right-hand side of (32) $r$ times and using (33), we obtain the right-hand side of (30). This completes the proof.
\end{proof}
\begin{example}
For the monomials $v_0$ and $\tau(v_0)$ used in Example 2.7 (i), the identity (30) also gives the following two different relations 
between MHZVs and (29), which are similar to (24) and (25): 
by taking $v=v_0$ in (30), 
\begin{equation*}
\begin{aligned}
&\sum_{\begin{subarray}{c}r_1+\cdots+r_p=r\\
r_i\ge0\end{subarray}}
\left\{\prod_{i=1}^{p}\binom{k_{i}+r_{i}-1}{r_{i}}\right\}
\zeta\left(z_1(k_1+r_1)\left\{\prod_{i=2}^{p}z_{0}(k_i+r_i)\right\}; 
\alpha\right)\\
=
&\sum_{\begin{subarray}{c}r_1+\cdots+r_q=r\\
r_i\ge0\end{subarray}}
H^{*}_{(\{r_i\}_{i=1}^{q})}
\left(\left\{\prod_{i=1}^{q-1}
z_{c^{'}_{i-1}}(1)\right\}z_{c^{'}_{q-1}}(2);\alpha\right)
\end{aligned}
\end{equation*}
and, by taking $v=\tau(v_0)$ in (30), 
\begin{equation*}
\begin{aligned}
&\sum_{\begin{subarray}{c}r_1+\cdots+r_q=r\\
r_i\ge0\end{subarray}}
(1+r_q)\zeta\left(\left\{\prod_{i=1}^{q-1}z_{c^{'}_{i-1}}(1+r_i)\right\}z_{c^{'}_{q-1}}(2+r_q);\alpha\right)\\
=
&\sum_{\begin{subarray}{c}r_1+\cdots+r_p=r\\
r_i\ge0\end{subarray}}
H^{*}_{(\{r_i\}_{i=1}^{p})}
\left(z_1(k_1)\left\{\prod_{i=2}^{p}z_{0}(k_i)\right\};\alpha\right)
\end{aligned}
\end{equation*}
for all $r\ge0$, $\alpha\in\mathbb{C}$ with $\mathrm{Re}(\alpha)>0$.
\end{example}
\begin{remark}
Coppo \cite{c}, Coppo and Candelpergher \cite{cc}, \'{E}mery \cite{e}, Hasse \cite{h} proved relations between the case $c_i=0$ 
(or $c_i=1$) ($i=1,\ldots,p-1$) of (28) and the single Hurwitz(--Lerch) zeta values 
$\zeta(z_1(k_1);\alpha)$, $\sum_{m=0}^{\infty}z^{m}(m+\alpha)^{-k}$. 
We proved in \cite{ig2007} relations between the above case of (28) 
with $z=1$ and the case $c_i=1$ ($i=1,\ldots,p-1$) of (27): 
see \cite[Proposition 2 and its examples]{ig2007} and \cite[\textbf{(R3)}]{i2018}. 
See also \cite[Note 2]{i2020}. 
\end{remark}
\begin{note} 
(i) The present paper is a revised version of preprints of mine which were 
distributed in October 2015. Revised versions of the preprints were submitted to 
many journals in June 2016--2021. For instance, one was submitted to 
a journal on October 30, 2017 and rejected on December 24, 2018. 
Another was submitted to a journal on March 9, 2019 and rejected 
on August 13, 2020. See also \cite[Note 2 (iii)--(v)]{i2020}. 
\par 
(ii) I give some notes related to the present research. 
In \cite{i2018}, I studied relations among two-, three- and four-parameter extensions of the multiple series (2), (27) and (28) by using the hypergeometric identities of Andrews \cite[Theorem 4]{an}, Krattenthaler and Rivoal \cite[Proposition 1]{kr}. This study of mine shows the potentiality of these extensions. The paper \cite{i2018} is a revised version of my manuscript submitted to the Nagoya Mathematical Journal on March 9, 2015. 
(A preprint of the manuscript was distributed in February 2015. 
I had already written most of the contents of \cite{i2018} in the manuscript 
and the preprint.) The above multi-parameter extensions were already studied 
in the preprint, the manuscript and other manuscripts of mine written in 2014--2016. 
See also Note 2 of \cite{i2020}. I note also that the case $v=\prod_{i=1}^{p}z_{1}(k_i)$ of (30) was proved in the preprint and the manuscript submitted 
to Nagoya Math. J. in the same way as in the proof of (30): 
the case $v=\prod_{i=1}^{p}z_{1}(k_i)$ of the proof of (30) is just the proof written therein. 
One of the most important ingredients of the proof of (30) is the duality formula (32). 
As regards (32), I stated the case $v=\prod_{i=1}^{p}z_{1}(k_i)$ of it in my talk at Seminar on Analytic Number Theory, Graduate School of Mathematics, Nagoya University, Japan, 
February 13, 2008. (See also \cite[Acknowledgments on p.~578]{i2009} 
and \cite[Note on pp.~21--22]{i202206}. The following people were parts of the audience of my talk: Kohji Matsumoto, Yoshio Tanigawa, Takashi Nakamura, Yoshitaka Sasaki.) Therefore the identities (30), (32) and the proof of (30) are extensions of my previous works on 
MHZVs. In my talk, I pointed out also a similarity between the multiple series on 
the right-hand side of the case $v=\prod_{i=1}^{p}z_{1}(k_i)$ of (32) and the Newton series studied 
by Kawashima in \cite{kaw2009}. Indeed, as explained in my talk, the Newton series 
has the factor $(\alpha)_{m_q}/{m_q}!$ in its summand and the multiple series on 
the right-hand side has the inverse ${m_q}!/(\alpha)_{m_q}$. 
I think that it is interesting to study this similarity further. 
From my previous works \cite{i2009} and \cite{i2011}, 
it can be seen that the identities 
\begin{equation*}
\begin{aligned}
&\frac{(-1)^r}{r!}\frac{\mathrm{d}^r}{\mathrm{d}\alpha^r}
\left(\sum_{\begin{subarray}{c}0{\le}m_1<_{c_1}\cdots<_{c_{p-1}}m_p<\infty\end{subarray}}
\frac{(m_p+1)!}{(\alpha)_{m_p+1}}
\left\{\prod_{i=1}^{p}\frac{1}{(m_i+1)^{k_i}}\right\}\right)\\
=
&\sum_{m_p=0}^{\infty}\frac{(m_p+1)!}{(\alpha)_{m_p+1}}\\
&\times\left(\sum_{\begin{subarray}{c}0{\le}l_1{\le}\cdots{\le}l_r{\le}m_p\end{subarray}}
\prod_{i=1}^{r}\frac{1}{l_i+\alpha}\right)
\left(\sum_{\begin{subarray}{c}0{\le}m_1<_{c_1}\cdots<_{c_{p-2}}
m_{p-1}<_{c_{p-1}}m_p\end{subarray}}
\prod_{i=1}^{p}\frac{1}{(m_i+1)^{k_i}}\right)\\
=
&\textrm{ a $\mathbb{Z}$-linear combination of (28) with $z=1$}
\end{aligned}
\end{equation*}
hold for all $r\ge0$, $\alpha\in\mathbb{C}$ with $\mathrm{Re}(\alpha)>0$. 
I note that, using my identities (34), one can get an explicit expression of the above 
$\mathbb{Z}$-linear combination without calculating the product of finite multiple harmonic sums. 
\par 
(iii) In \cite[Remark 7 (i)]{i2018}, I gave a new proof of Hoffman's identity 
$\zeta(\{1\}^k, l+2)=\zeta(\{1\}^l, k+2)$ ($0{\le}k,l\in\mathbb{Z}$; 
\cite[Theorem 4.4]{ho}), which is the duality formula for $\zeta(\{1\}^k,l+2)$. 
Here $\zeta(\{k_i\}_{i=1}^n):=\sum_{\begin{subarray}{c}0<m_1<\cdots<m_n<\infty\end{subarray}}m_1^{-k_1}{\cdots}m_n^{-k_n}$ 
(i.e., MZV) and $\{1\}^n:=\{1\}_{i=1}^n$. My proof is based on the hypergeometric identities \cite[Theorem 4]{an} and \cite[Proposition 1 (i)]{kr}; 
therefore it can be regarded as a hypergeometric proof of the duality formula. 
It is interesting to generalize my proof to a proof of the duality formula 
for all MZVs in appropriate ways. 
\end{note}
\begin{corrections2} 
(i) Page 223, line 7: ``the idea" should be ``our idea". 
(ii) Page 223, line 15 from the bottom and page 234, line 12: ``former" should be ``previous". (iii) Page 235, line 7: ``sort" should be ``kind". 
\end{corrections2}
\begin{corrections1} 
(i) Page 575, lines 2--3 from the bottom: ``what we noted" should be 
``a note on". (ii) Page 578, line 23 from the bottom: ``March 2007" should be 
``February 3, 2007". 
\end{corrections1}
\begin{flushleft}
Nagoya, Japan\\
\textit{E-mail address}: masahiro.igarashi2018@gmail.com
\end{flushleft}
\end{document}